\documentclass[a4paper,oneside,11pt, reqno]{amsproc}
\usepackage{amsthm}
\usepackage{amsfonts}
\usepackage{amsmath}
\usepackage{amssymb}
\usepackage{mathrsfs}
\usepackage{epsfig}
\usepackage{graphicx}
\usepackage{epstopdf}
\usepackage{tikz-cd}
\usepackage{color}
\usepackage{hyperref}
\usepackage{ifthen}
\usepackage{float}
\usepackage{comment}
\makeatletter
\@namedef{subjclassname@2010}{%
\textup{2010} Mathematics Subject Classification}
\makeatother

\newtheorem{theorem}{Theorem}[section]
\newtheorem{lemma}[theorem]{Lemma}

\theoremstyle{remark}
\newtheorem{remark}[theorem]{Remark}

\theoremstyle{definition}

\newcommand{\bq}{\begin{equation}}
\newcommand{\eq}{\end{equation}}
\newcommand{\beqn}{\begin{eqnarray*}}
\newcommand{\eeqn}{\end{eqnarray*}}
\newcommand{\beq}{\begin{eqnarray}}
\newcommand{\eeq}{\end{eqnarray}}

\newcommand{\rar}{\rightarrow}

\newcommand{\bc}{\begin{centre}}
\newcommand{\ec}{\end{centre}}

\newcommand{\ba}{\begin{array}}
\newcommand{\ea}{\end{array}}

\newcommand{\inp}[2]{\langle{#1},\,{#2} \rangle}

\renewcommand{\Delta}{{\nabla}}

\begin{document}
\title[Brown-Halmos type characterization for the tetrablock]{Brown-Halmos type characterization for the tetrablock}
\author[S. Jain]{Shubham Jain}
\address [S. Jain]{Department of Mathematics, Indian Institute of Technology Guwahati, Guwahati 781039, India}
\email{shubjainiitg@iitg.ac.in, shubhamjain4343@gmail.com}
\author[S. Kumar ]{Surjit Kumar}
\address[S. Kumar]{Department of Mathematics, Indian Institute of Technology Madras, Chennai 600036, India} 
\email{surjit@iitm.ac.in}
\author[M. K. Mal]{Milan Kumar Mal}
\address[M. K. Mal]{Department of Mathematics, Indian Institute of Technology Madras, Chennai 600036, India}
 \email{ma21d018@smail.iitm.ac.in; milanmal1702@gmail.com }
\author[P. Pramanick]{Paramita Pramanick}
\address[P. Pramanick]{Statistics and Mathematics Unit, Indian Statistical Institute Kolkata, Kolkata 700108, India}
\email{paramitapramanick@gmail.com}

\date{}

\begin{abstract}
In this note, we obtain a Brown–Halmos type characterization for Toeplitz operators on the Hardy space associated with the tetrablock. As an application, we show that the zero operator is the only compact Toeplitz operator.
\end{abstract}

\maketitle
\section{Introduction}
A bounded linear operator defined on the classical Hardy space $H^2(\mathbb D)$ of the unit disc $\mathbb D$ by $T_\phi f:= P_{H^2(\mathbb{D})}(\phi f)$ is called a {\it Toeplitz operator}, where $\phi \in L^\infty(\mathbb T) $, $f \in H^2(\mathbb D)$ and $P_{H^2(\mathbb D)}$ denotes the orthogonal projection of $L^2(\mathbb T)$ onto $H^2(\mathbb D)$. The study of Toeplitz operators on function spaces has a rich history, beginning with the pioneering work of Brown and Halmos in \cite{BH1964}. Since then, extensive explorations have been made in the same direction for domains such as polydisc, unit ball, symmetrized bidisc, Hartogs triangle, and bounded symmetric domains; for instance, see \cite{ BDS2021, DJ1977,  DSZ2012,  GR2024, JP2024, MSS2018, Up1983}. 
 Recently, a study of Toeplitz operators on certain holomorphic function spaces on the proper images of bounded symmetric domains has emerged in \cite{GN2023, GR2024}. One of the main results studied by Brown and Halmos is an algebraic characterization of Toeplitz operators on $H^2(\mathbb D)$, commonly known as the Brown-Halmos characterization (see \cite[Theorem 6]{BH1964}). This characterization has been extended to various domains, including the above mentioned domains.  In \cite{BH1964}, Brown and Halmos also observed that the only compact Toeplitz operator on $H^2(\mathbb D)$ is the zero operator.
 In this paper, we study analogous results in the context of the tetrablock. 

The tetrablock $\mathbb E$ is given by \beqn \mathbb E = \{(z_1, z_2, z_3) \in \mathbb C^3 : 1-z_1z-z_2w+z_3zw \neq 0~\mbox{for all~} z,w \in \overline{\mathbb D}\}.
\eeqn
 The Shilov boundary $S_\mathbb E$ of $\mathbb E$ is the set $ S_\mathbb E=\{(z_1,z_2,z_3)\in \mathbb C^3: z_1=\bar{z}_2 z_3, \; |z_3|=1, \; |z_2|\leq 1\}$ (see \cite[Theorem 7.1]{AWY2007}).
Note that if $z=(z_1,z_2,z_3)\in S_{\mathbb E}$, then $z_2=\bar{z}_1 z_3.$

The tetrablock is first studied extensively in \cite{AWY2007}. Further, tetrablock has been studied in many contexts (see \cite{MT2013, BT2014}). An explicit expression of reproducing kernel for the Bergman space on the tetrablock can be found in \cite[Corollary 2]{MT2013}. 
A notion of Hardy space on a class of domains consisting tetrablock can be found in \cite{GR2024}. The Hardy space $H^2(\mathbb{E})$ on the tetrablock $ \mathbb{E}$ was recently studied in \cite[Section 3.2]{BCJ2024}, following the approach in \cite{MRZ2013}. This construction uses the fact that  $\mathbb{E}$  is the image of a type-II classical Cartan domain in three dimensions under a proper holomorphic map. 
A concrete formula for the reproducing kernel (Szeg\"o kernel) of $H^2(\mathbb E)$ is given in \cite[Equation 9]{BCJ2024}.  

By a commuting $3$-tuple $T=(T_1, T_2, T_3)$ on a complex separable Hilbert space $\mathcal{H}$,  we mean that $T_1, T_2, T_3$ are mutually commuting bounded linear operators on $\mathcal{H}.$ For a subset $\mathcal{S}$ of $\mathcal{H}$, the notation $\bigvee \mathcal{S}$ stands for the closed linear span of $\mathcal{S}$. The main result of this paper establishes a Brown-Halmos type characterization of Toeplitz operators on the Hardy space $H^2(\mathbb E):$
\begin{theorem}\label{Brown-Halmos}
    Let $\boldsymbol{T_z}=(T_{z_1}, T_{z_2}, T_{z_3})$ be the commuting $3$-tuple of multiplication operators on $H^2(\mathbb E)$ by the coordinate functions $z_1,z_2,z_3$. Then a necessary and sufficient condition for a bounded linear operator $T$ on $H^2(\mathbb E)$ to be a Toeplitz operator are the following conditions: $T T_{z_1}=T_{z_2}^* T T_{z_3}$, $T T_{z_2}=T_{z_1}^* T T_{z_3}$, and $T_{z_3}^* T T_{z_3}=T.$
\end{theorem}
 Section 2 is devoted to a preliminary discussion of the Hardy space of the tetrablock and associated Toeplitz operators.
In Section 3, we present a proof of the main result. 
We show in Section 4 that any compact Toeplitz operator on $H^2(\mathbb E)$ is the zero operator.

\section{Preliminaries}
In this section, we recall the notion of the Hardy space on tetrablock introduced in \cite[Subsection~3.2]{BCJ2024} (see also \cite[Example~3.11]{GR2024}) and discuss the associated Toeplitz operators. 

\subsection{Hardy space on tetrablock}
Consider the type-II Cartan domain of dimension 3 by $$\mathfrak{R}_{\mathrm{II}} = \left\{ (z_1, z_2, z_3) \in \mathbb{C}^3 : \left\| \begin{pmatrix} z_1 & z_3 \\ z_3 & z_2 \end{pmatrix} \right\| < 1 \right\}.$$ 

We first recall the Hardy space on $\mathfrak{R}_{\mathrm{II}}$ (see \cite[p. 521]{HM1969}). 
Let $K$ denote the group of all linear biholomorphic automorphisms of the domain $\mathfrak{R}_{\mathrm{II}}$.
Let $d\Theta$ be the unique $K$-invariant probability measure supported on the Shilov boundary $S_{\mathfrak{R}_{\mathrm{II}}}$ of $\mathfrak{R}_{\mathrm{II}}.$ Let $L^2(S_{\mathfrak{R}_{\mathrm{II}}})$ be the Hilbert space consisting of square integrable functions $S_{\mathfrak{R}_{\mathrm{II}}}$ with respect to $d\Theta.$   

Note that the $L^2$-inner product is $K$-invariant, that is, $\inp{f\circ k}{g\circ k}_{L^2(S_{\mathfrak{R}_{\mathrm{II}}})}=\inp{f}{g}_{L^2(S_{\mathfrak{R}_{\mathrm{II}}})},$ for all $k\in K$.

The Hardy space $H^2(\mathfrak{R}_{\mathrm{II}})$ is given by $$\left\{f\in \mathrm{Hol}(\mathfrak{R}_{\mathrm{II}}):\|f\|_{H^2(\mathfrak{R}_{\mathrm{II}})}=\sup_{0<r<1} \left( \int_{S_{\mathfrak{R}_{\mathrm{II}}}} |f(rz)|^2 d\Theta(z)\right)^{\frac{1}{2}} <\infty\right \}.$$
 For every function $f\in H^2(\mathfrak{R}_{\mathrm{II}})$, its radial limit $\tilde{f}\in L^2(S_{\mathfrak{R}_{\mathrm{II}}})$ and $\|f\|_{H^2(\mathfrak{R}_{\mathrm{II}})}=\|\tilde{f}\|_{L^2(S_{\mathfrak{R}_{\mathrm{II}}})}$ (see \cite[Theorem 6]{HM1969}). Thus, there is an isometric embedding of $H^2(\mathfrak{R}_{\mathrm{II}})$ into $L^2(S_{\mathfrak{R}_{\mathrm{II}}})$. In this note, we do not distinguish between these two identifications.

Let $\sigma : \mathbb C^3 \rar \mathbb C^3$ be the involution given by $$\sigma(z_1, z_2, z_3)=(z_1, z_2, -z_3), \quad (z_1, z_2, z_3) \in \mathbb C^3.$$ 
Since $\big[\begin{smallmatrix} 1 & 0 \\ 0 & -1\end{smallmatrix}\big]\big[\begin{smallmatrix} z_{1} & z_3 \\ z_3 & z_{2}\end{smallmatrix}\big]\big[\begin{smallmatrix} 1 & 0 \\ 0 & -1\end{smallmatrix}\big]=\big[\begin{smallmatrix} z_{1} & -z_3 \\ -z_3 & z_{2}\end{smallmatrix}\big],$ the domain $\mathfrak R_{II}$ is $\sigma$-invariant.
Let $\phi$ be the map given by $\phi=(\phi_1, \phi_2, \phi_3),$ where 
\beq \label{phi definition}
\phi_1(z)=z_1,\; \phi_2(z)=z_2,\; \phi_3(z)=z_1z_2-z^2_3, \quad z=(z_1, z_2, z_3) \in \mathbb C^3.
\eeq
Note that $\phi$ is a proper map of multiplicity $2$ from $\mathfrak R_{II}$ onto $\mathbb E$ (see \cite[Theorem 5.1]{R1982}). Moreover, $\phi$ is $\sigma$-invariant, that is, $\phi \circ \sigma=\phi.$

Let $L^2_{-}(S_{\mathfrak{R}_{\mathrm{II}}})= \left\{f \in L^2(S_{\mathfrak{R}_{\mathrm{II}}}) : f =-f\circ\sigma \right\}.$ Then $L^2_{-}(S_{\mathfrak{R}_{\mathrm{II}}})$ is a closed subspace of $L^2(S_{\mathfrak{R}_{\mathrm{II}}}).$ 
Similarly, $H^2_{-}(\mathfrak{R}_{II})=\left\{ f \in H^2(\mathfrak{R}_{II}) : f =-f \circ \sigma\right\}$ is a closed subspace of $H^2(\mathfrak{R}_{II}).$

The Hardy space $H^2(\mathbb E)$ of the tetrablock is defined as
$$H^2(\mathbb E)=\left\{f \in \mathrm{Hol}(\mathbb E) : \|\mathcal J_{\phi} f \circ \phi\|_{H^2(\mathfrak{R}_{II})} < \infty\right\},$$ where $\mathcal{J}_\phi$ is the determinant of the Jacobian matrix of the proper map $\phi.$ We fix the notation $\boldsymbol{T_z}$ for the tuple $(T_{z_1}, T_{z_2}, T_{z_3})$ of operators of multiplications by the coordinate functions $z_1, \; z_2, \; z_3$ on $H^2(\mathbb E).$

The map $\Psi:H^2(\mathbb E) \rar H^2_{-}(\mathfrak{R}_{II})$ defined as 
\beq \label{psihardyhardy}
\Psi(f)=\mathcal{J}_{\phi}f \circ \phi, \quad f \in H^2(\mathbb E),
\eeq 
is a unitary map (see \cite[Theorem~1.3]{BCJ2024}, \cite[Lemma~3.2]{GR2024}). 

Let $S_\mathbb E$ denote the Shilov boundary of the tetrablock $\mathbb E.$ Then $\phi(S_{\mathfrak{R}_{\mathrm{II}}})=S_{\mathbb E}.$ We consider $d\Theta_{\mathbb E}$ as the measure supported on $S_{\mathbb E}$ given by the formula \beq \label{eq 6} \int_{S_\mathbb E} f d\Theta_{\mathbb E}=\frac{1}{C}\int_{S_{\mathfrak{R}_{\mathrm{II}}}} f \circ \phi  |\mathcal{J}_{\phi}|^2 d\Theta,
 \eeq
 where $C$ is the constant such that $\|1\|_{L^2(d\Theta_{\mathbb E})}=1.$

 Let $L^2(S_{\mathbb E})$ be the set of all square integrable functions on $S_\mathbb E$ with respect to the measure $d\Theta_{\mathbb E}.$ By \cite[Lemma 3.15]{GR2024}, $H^2(\mathbb E)$ is isometrically embedded in $L^2(S_{\mathbb E}),$ that is, each $f \in H^2(\mathbb E)$ has boundary value in $L^2(S_{\mathbb E}).$ 

Consider the map
\beq \label{eq 12}
\tilde{\Psi} : L^2(S_\mathbb E) \rar L_{-}^2(S_{\mathfrak{R}_{\mathrm{II}}}) \mbox{ defined by } \tilde{\Psi}(f)=\mathcal{J}_{\phi} f \circ \phi, \quad f \in L^2(S_\mathbb E).
\eeq

Clearly, $\tilde{\Psi}$ is well defined and an isometry. Since $\mathcal J_{\phi}$ vanishes only on a set of measure zero (with respect to the measure $d\Theta$), $f_1=\mathcal{J}_{\phi}^{-1}\cdot f$ is $\sigma$-invariant measurable function with respect to $d\Theta$ for all $f \in L^2_{-}(S_{\mathfrak{R}_{\mathrm{II}}}).$ 

Thus by \cite[Remark 2.2(3)]{GN2023}, there exists a bounded measurable function $\hat{f}$ on $S_\mathbb E$ such that $f_1=\hat{f}\circ \phi.$ Therefore, $f=\mathcal{J}_\phi \hat{f}\circ \phi$ yields that $\tilde\Psi$ is unitary (see \cite[Lemma 3.14]{GR2024}).

\subsection{Toeplitz operator on $H^2(\mathbb E)$}
Let $L^\infty(S_{\mathfrak{R}_{\mathrm{II}}})$ be the space of all essentially bounded measurable functions with respect to the measure $d\Theta.$ 
The set of all $\sigma$-invariant functions in $ L^\infty(S_{\mathfrak{R}_{\mathrm{II}}})$ is denoted by $L_{+}^\infty(S_{\mathfrak{R}_{\mathrm{II}}}).$
Let $L^\infty(S_{\mathbb E})$ be the space of all essentially bounded measurable functions on $S_\mathbb E$ with respect to the measure $d\Theta_\mathbb E.$ 
Then the space $L^\infty(S_{\mathbb E})$ is isometrically $*$-isomorphic to $L_{+}^\infty(S_{\mathfrak{R}_{\mathrm{II}}})$ via the map $\hat{u} \mapsto \hat{u}\circ \phi$ (see \cite[Subsection 3.3]{GR2024}).

The Toeplitz operator $\mathcal{T}_u$ on $H^2_{-}(\mathfrak{R}_{\mathrm{II}})$ with the symbol $u\in L_{+}^\infty(S_{\mathfrak{R}_{\mathrm{II}}})$ is defined as 
\beqn (\mathcal{T}_uf)(z)=P_{H^2_{-}(\mathfrak{R}_{\mathrm{II}})}(uf), \quad f \in H^2_{-}(\mathfrak{R}_{II}),\eeqn 
where $P_{H^2_{-}(\mathfrak{R}_{\mathrm{II}})}$ is the orthogonal projection of $L_{-}^2(S_{\mathfrak{R}_{\mathrm{II}}})$ onto $H_{-}^2(\mathfrak{R}_{\mathrm{II}}).$ 
Since $uf\in L_{-}^2(S_{\mathfrak{R}_{\mathrm{II}}})$ for any $f\in  H_{-}^2(\mathfrak{R}_{\mathrm{II}})$ and $u\in L_{+}^\infty(S_{\mathfrak{R}_{\mathrm{II}}}),$  the operator $\mathcal{T}_u$ is well defined. 
If $P_{H^2(\mathbb E)}$ denote the orthogonal projection of $L^2(S_\mathbb E)$ onto $H^2(\mathbb E)$ then the Toeplitz operator $T_{\hat{u}}$ on $H^2(\mathbb E)$ for $\hat{u}\in L^\infty(S_{\mathbb E})$ is given by 
\beqn (T_{\hat{u}}\hat{f})(z)=P_{H^2(\mathbb E)}(\hat{u}\hat{f}), \quad \hat{f} \in H^2(\mathbb E).\eeqn
 
Moreover, the following diagram commutes (see \cite[Lemma 4.2]{GR2024}).
$$\begin{tikzcd}
L^2(S_{\mathbb E})\arrow{r}{\tilde{\Psi}} \arrow[swap]{d}{P_{H^2(\mathbb E)}} &  L_{-}^2(S_{\mathfrak{R}_{\mathrm{II}}}) \arrow{d}{P_{H^2_{-}(\mathfrak{R}_{\mathrm{II}})}} \\
 H^2(\mathbb E) \arrow{r}{\Psi} &H^2_{-}(\mathfrak{R}_{\mathrm{II}})
\end{tikzcd}$$
That is, $P_{H^2_{-}(\mathfrak{R}_{\mathrm{II}})}\tilde{\Psi}=\Psi P_{H^2(\mathbb E)}.$ This yields \beq \label{toeplitzrelation} \mathcal{T}_{\hat{u}\circ \phi}\Psi=\Psi T_{\hat{u}},\; \hat{u} \in L^{\infty}(S_{\mathbb E}).\eeq

\section{Proof of the main theorem }
This section is devoted to the proof of Theorem \ref{Brown-Halmos}. 
To that end, we first establish a few auxiliary lemmas, starting with the following discussion. 

For $i=1,2,3$, let $N_{\phi_i}$
be the multiplication operator on $L^2(S_{\mathfrak{R}_{\mathrm{II}}})$ by the function $\phi_i.$ Then $\boldsymbol{N_{\phi}}=(N_{\phi_1}, N_{\phi_2}, N_{\phi_3})$ is a commuting $3$-tuple of normal operators on $L^2(S_{\mathfrak{R}_{\mathrm{II}}}).$ 
Note that for any $f\in L^2_{-}(S_{\mathfrak{R}_{\mathrm{II}}}),$ $((\phi_i  f) \circ \sigma)(z)= \phi_i(\sigma(z)) f(\sigma(z))=-\phi_i(z)f(z)$ for $i=1,2,3.$  Thus, $\phi_{i} f\in L_{-}^2(S_{\mathfrak{R}_{\mathrm{II}}}).$ Similarly, $\bar{\phi_i} f\in L_{-}^2(S_{\mathfrak{R}_{\mathrm{II}}})$ for any $f\in L^2_{-}(S_{\mathfrak{R}_{\mathrm{II}}})$ and $i=1,2,3.$ Subsequently, $L_{-}^2(S_{\mathfrak{R}_{\mathrm{II}}})$ is a reducing subspace for $\boldsymbol{N_{\phi}}.$  Set $\mathcal{M}_{\phi_i}=N_{\phi_i}|_{L^2_{-}(S_{\mathfrak{R}_{\mathrm{II}}})}$ for $i=1, 2, 3.$ Therefore, $\boldsymbol{\mathcal{M}_\phi}=(\mathcal{M}_{\phi_1}, \mathcal{M}_{\phi_2}, \mathcal{M}_{\phi_3})$ is a commuting normal $3$-tuple on $L^2_{-}(S_{\mathfrak{R}_{\mathrm{II}}}).$

Let $\boldsymbol{M_z}=(M_{z_1}, M_{z_2}, M_{z_3})$ be the tuple of multiplication operators by the coordinate functions $z_1,z_2, z_3$ on $L^2(S_\mathbb E).$ Recall that if $z=(z_1,z_2,z_3)\in S_{\mathbb E}$, then $|z_3|=1, \; z_1=\bar{z}_2 z_3$ and $z_2=\bar{z}_1 z_3$ and hence $\boldsymbol{M_z}$ satisfies the following relations:
\beq \label{M-relations} 
M_{z_3}^* M_{z_3}=M_{z_3} M_{z_3}^*=I,\; M_{z_1}=M_{z_2}^* M_{z_3}, \mbox{ and } M_{z_2}=M_{z_1}^* M_{z_3}.
\eeq

It follows from \eqref{eq 12} that $\tilde{\Psi} M_{z_i}= \mathcal{M}_{\phi_i} \tilde{\Psi}$ for $i=1,2,3,$ that is, $\boldsymbol{M_z}$ is unitarily equivalent to $\boldsymbol{\mathcal{M}_\phi}.$ This with \eqref{M-relations} yields that 
\beq \label{relation mathcal M}
\mathcal{M}_{\phi_3}^* \mathcal{M}_{\phi_3}=I, \mathcal{M}_{\phi_1}=\mathcal{M}_{\phi_2}^* \mathcal{M}_{\phi_3} \mbox{ and } \mathcal{M}_{\phi_2}=\mathcal{M}_{\phi_1}^* \mathcal{M}_{\phi_3}.
\eeq

Let $\boldsymbol{\mathcal{T}_\phi} = (\mathcal{T}_{\phi_1}, \mathcal{T}_{\phi_2}, \mathcal{T}_{\phi_3})$, where each $\mathcal{T}_{\phi_i}$ is the Toeplitz operator on $H^2_{-}(\mathfrak{R}_{\mathrm{II}})$ with symbol $\phi_i$, for $i = 1, 2, 3$, as defined in~\eqref{phi definition}. The preceding discussion leads to the following result.

\begin{lemma} \label{lemma 1}
    $\boldsymbol{\mathcal{T}_\phi}$ is a commuting subnormal $3$-tuple on $H^2_{-}({\mathfrak{R}_{\mathrm{II}}})$ with minimal normal extension $\boldsymbol{\mathcal{M}_\phi}.$ Moreover, $\boldsymbol{\mathcal{T}_\phi}$ satisfies the following: $\mathcal{T}_{\phi_1}=\mathcal{T}_{\phi_2}^* \mathcal{T}_{\phi_3}$, $\mathcal{T}_{\phi_2}=\mathcal{T}_{\phi_1}^* \mathcal{T}_{\phi_3},$ and $ \mathcal{T}_{\phi_3}^* \mathcal{T}_{\phi_3}=I.$
\end{lemma}
\begin{proof}
Since $H^2_{-}(\mathfrak{R}_{\mathrm{II}})$ is a joint invariant subspace of $L^2_{-}(S_{\mathfrak{R}_{\mathrm{II}}})$ for $\boldsymbol{\mathcal{M}_{\phi}},$ $\boldsymbol{\mathcal{T}_\phi}$ 
is a subnormal tuple.  We now show that $\boldsymbol{\mathcal{M}_\phi}$ is the minimal normal extension of $\boldsymbol{\mathcal{T}_\phi}.$ For this, we need to show that $$L_{-}^2(S_{\mathfrak{R}_{\mathrm{II}}})=\bigvee \{\mathcal{M}_{\phi_1}^{*\alpha_1}\mathcal{M}_{\phi_2}^{*\alpha_2}\mathcal{M}_{\phi_3}^{*\alpha_3}h: h\in H_{-}^2(\mathfrak{R}_{\mathrm{II}}), \alpha_i\in \mathbb Z_{+}, i=1, 2, 3\}.$$
If not, then $\bigvee \{\mathcal{M}_{\phi_1}^{*\alpha_1}\mathcal{M}_{\phi_2}^{*\alpha_2}\mathcal{M}_{\phi_3}^{*\alpha_3}h: h\in H_{-}^2(\mathfrak{R}_{\mathrm{II}}), \alpha_i\in \mathbb Z_{+}, i=1, 2, 3\}$ is a proper subspace of $L_{-}^2(S_{\mathfrak{R}_{\mathrm{II}}})$, reducing under $\boldsymbol{\mathcal{M}_{\phi}}.$
Since  $\boldsymbol{\mathcal{M}_{\phi}}$ is unitarily equivalent to $\boldsymbol{M_z},$ the closed subspace $\bigvee\{M_{z_1}^{*\alpha_1}M_{z_2}^{*\alpha_2}M_{z_3}^{*\alpha_3}\hat{h}: \hat{h}\in H^2(\mathbb E)\;, \alpha_i \in \mathbb Z_+, i=1,2,3\}$ is a proper reducing subspace of $L^2(S_\mathbb E)$ under $\boldsymbol{M_z}.$ This contradicts the fact that $\boldsymbol{M_z}$ is the minimal normal extension of $\boldsymbol{T_z}.$ Thus $\boldsymbol{\mathcal{M}_\phi}$ is the minimal normal extension of $\boldsymbol{\mathcal{T}_\phi}.$
 
For any $h,g\in H^2_{-}(\mathfrak{R}_{\mathrm{II}}),$
\beqn \inp{\mathcal{T}_{\phi_1}h}{g}_{H^2_{-}(\mathfrak{R}_{\mathrm{II}})} &=& \inp{\mathcal{M}_{\phi_1} h}{g}_{L^2_{-}(S_{\mathfrak{R}_{\mathrm{II}}})}\\
&\overset{\eqref{relation mathcal M}}{=}& \inp{\mathcal{M}_{\phi_2}^* \mathcal{M}_{\phi_3}h}{g}_{L^2_{-}(S_{\mathfrak{R}_{\mathrm{II}}})}\\
&=& \inp{\mathcal{T}_{\phi_2}^* \mathcal{T}_{\phi_3}h}{g}_{H^2_{-}(\mathfrak{R}_{\mathrm{II}})}.
\eeqn
Similarly, one can verify that $\mathcal{T}_{\phi_2}=\mathcal{T}_{\phi_1}^* \mathcal{T}_{\phi_3}$ and $ \mathcal{T}_{\phi_3}^* \mathcal{T}_{\phi_3}=I.$ This completes the proof.
\end{proof}

\begin{remark}
   The commuting tuple $\boldsymbol{T_z}=(T_{z_1}, T_{z_2}, T_{z_3})$ defined on $H^2(\mathbb E)$ is unitarily equivalent to the tuple $\boldsymbol{\mathcal{T}_\phi} = (\mathcal{T}_{\phi_1}, \mathcal{T}_{\phi_2}, \mathcal{T}_{\phi_3})$ on $H^2_{-}({\mathfrak{R}_{\mathrm{II}}})$. By Lemma \ref{lemma 1}, we have $T_{z_1}=T_{z_2}^* T_{z_3},\; T_{z_2}=T_{z_1}^* T_{z_3},$ and $T_{z_3}^* T_{z_3}=I.$
\end{remark}

\begin{lemma}\label{lemma 2}
    For any $f\in \mathrm{span}\{\mathcal{M}_{\phi_1}^{*\alpha_1}\mathcal{M}_{\phi_2}^{*\alpha_2}\mathcal{M}_{\phi_3}^{*\alpha_3}h: h\in H_{-}^2(\mathfrak{R}_{\mathrm{II}}), \alpha_i\in \mathbb Z_{+}, i=1, 2, 3\}$, there is an integer $r\in \mathbb Z_{+}$ such that $\mathcal{M}_{\phi_3}^r f\in H_{-}^2(\mathfrak{R}_{\mathrm{II}}).$ 
\end{lemma}

\begin{proof}
  
Consider $f=\mathcal{M}_{\phi_1}^{*\alpha_1}\mathcal{M}_{\phi_2}^{*\alpha_2}\mathcal{M}_{\phi_3}^{*\alpha_3}h$ for some $h\in H^2_{-}(\mathfrak{R}_{\mathrm{II}})$ and $\alpha_1,\alpha_2, \alpha_3\in \mathbb Z_+.$ Choose $r\in \mathbb Z_{+}$ such that $r>\alpha_1+\alpha_2+\alpha_3.$ Then by \eqref{relation mathcal M}, we have
\beqn
    \mathcal{M}_{\phi_3}^{r}\mathcal{M}_{\phi_1}^{*\alpha_1}\mathcal{M}_{\phi_2}^{*\alpha_2}\mathcal{M}_{\phi_3}^{*\alpha_3}h
&=&\mathcal{M}_{\phi_3}^{r-\alpha_3}\mathcal{M}_{\phi_1}^{*\alpha_1}\mathcal{M}_{\phi_2}^{*\alpha_2}\mathcal{M}_{\phi_3}^{\alpha_3}\mathcal{M}_{\phi_3}^{*\alpha_3}h\\
&=&\mathcal{M}_{\phi_3}^{r-\alpha_3}\mathcal{M}_{\phi_1}^{*\alpha_1}\mathcal{M}_{\phi_2}^{*\alpha_2}h\\
&=&\mathcal{M}_{\phi_3}^{r-\alpha_3-\alpha_2}\mathcal{M}_{\phi_1}^{*\alpha_1}\mathcal{M}_{\phi_1}^{\alpha_2}h\\
&=&\mathcal{M}_{\phi_3}^{r-\alpha_3-\alpha_2-\alpha_1}\mathcal{M}_{\phi_2}^{\alpha_1}\mathcal{M}_{\phi_1}^{\alpha_2}h.
\eeqn
This proves the lemma.
\end{proof}

The following result plays a key role in the proof of Theorem \ref{Brown-Halmos}.
\begin{lemma}\label{key lemma}
    Let $A$ be a bounded linear operator on $H_{-}^2(\mathfrak{R}_{\mathrm{II}})$ such that $A$ satisfies the following conditions:
    \beq \label{brown-halmos10} A\mathcal{T}_{\phi_1}=\mathcal{T}_{\phi_2}^* A \mathcal{T}_{\phi_3}, \;A\mathcal{T}_{\phi_2}=\mathcal{T}_{\phi_1}^*A \mathcal{T}_{\phi_3}, \text{and}\,\, \mathcal{T}_{\phi_3}^*A\mathcal{T}_{\phi_3}= A.\eeq 
    Then there exists a bounded linear operator $X$ on $L_{-}^2(S_{\mathfrak{R}_{\mathrm{II}}})$ 
    such that \beqn P_{H_{-}^2(\mathfrak{R}_{\mathrm{II}})}X|_{H_{-}^2(\mathfrak{R}_{\mathrm{II}})}=A, \; X\mathcal{M}_{\phi_i}=\mathcal{M}_{\phi_i}X, \mbox{ for } i=1,2,3 \mbox{ and } \|X\|=\|A\|. \eeqn
\end{lemma}
\begin{proof}
Consider a sequence $\{A_r\}_{r \in \mathbb Z_+}$ of operators defined on $L_{-}^2(S_{\mathfrak{R}_{\mathrm{II}}})$ given by $A_r=\mathcal{M}_{\phi_3}^{*r}AP_{H_{-}^2(\mathfrak{R}_{\mathrm{II}})}\mathcal{M}_{\phi_3}^{r}$. Let $V:=\mathrm{span}\{\mathcal{M}_{\phi_1}^{*\alpha_1}\mathcal{M}_{\phi_2}^{*\alpha_2}\mathcal{M}_{\phi_3}^{*\alpha_3}h: h\in H_{-}^2(\mathfrak{R}_{\mathrm{II}}), \alpha_i\in \mathbb Z_{+}, i=1, 2, 3\}.$ Since $\boldsymbol{\mathcal{M}_\phi}$ is the minimal normal extension of $\boldsymbol{\mathcal{T}_\phi}$ (see Lemma \ref{lemma 1}), $V$ is dense in $L_{-}^2(S_{\mathfrak{R}_{\mathrm{II}}})$.  By using Lemma \ref{lemma 2}, for any $g,h \in V,$ there exists $r_0\in \mathbb Z_{+}$ such that $\mathcal{M}_{\phi_3}^{r_0}g, \; \mathcal{M}_{\phi_3}^{r_0}h \in H_{-}^2(\mathfrak{R}_{\mathrm{II}}).$
 
   Therefore, for any $r\geq r_0,$
\begin{align*}
       \langle A_{r}g,  h\rangle_{L_{-}^2(S_{\mathfrak{R}_{\mathrm{II}}})}
       &= \Big\langle \mathcal{M}_{\phi_3}^{*r}AP_{H_{-}^2(\mathfrak{R}_{\mathrm{II}})}\mathcal{M}_{\phi_3}^{r}g,  h\Big\rangle_{L_{-}^2(S_{\mathfrak{R}_{\mathrm{II}}})}\\
       &= \Big\langle A \mathcal{M}_{\phi_3}^{r-r_0}\mathcal{M}_{\phi_3}^{r_0}g,  \mathcal{M}_{\phi_3}^{r-r_0}\mathcal{M}_{\phi_3}^{r_0}h\Big\rangle_{L_{-}^2(S_{\mathfrak{R}_{\mathrm{II}}})}\\
       &= \Big\langle A \mathcal{T}_{\phi_3}^{r-r_0}(\mathcal{M}_{\phi_3}^{r_0}g),  \mathcal{T}_{\phi_3}^{r-r_0}(\mathcal{M}_{\phi_3}^{r_0}h)\Big\rangle_{H_{-}^2(\mathfrak{R}_{\mathrm{II}})}\\
       &= \Big\langle \mathcal{T}_{\phi_3}^{*r-r_0}A \mathcal{T}_{\phi_3}^{r-r_0}(\mathcal{M}_{\phi_3}^{r_0}g),  (\mathcal{M}_{\phi_3}^{r_0}h)\Big\rangle_{H_{-}^2(\mathfrak{R}_{\mathrm{II}})}\\
       &\overset{\eqref{brown-halmos10}}{=} \Big\langle A (\mathcal{M}_{\phi_3}^{r_0}g),  (\mathcal{M}_{\phi_3}^{r_0}h)\Big\rangle_{H_{-}^2(\mathfrak{R}_{\mathrm{II}})}.
   \end{align*}
 
This shows that
\beqn 
    \lim_{r\to \infty}\langle A_{r}g,  h\rangle_{L_{-}^2(S_{\mathfrak{R}_{\mathrm{II}}})}=\Big\langle A (\mathcal{M}_{\phi_3}^{r_0}g),  (\mathcal{M}_{\phi_3}^{r_0}h)\Big\rangle_{H_{-}^2(\mathfrak{R}_{\mathrm{II}}).}
\eeqn
Note that $\|A_r\|\leq \|A\|$ for all $r\in \mathbb Z_+$ and the bi-linear form defined by $ \displaystyle B(g, h):= \lim_{r\to \infty}\langle A_{r}g,  h\rangle$ is bounded.  In fact, $|B(g, h)|\leq \|A\|\|g\|\|h\|.$
We now extend $B$ to a bounded bi-linear form on $L_{-}^2(S_{\mathfrak{R}_{\mathrm{II}}})$ and denote it by $B$ again.
Therefore, there exists a unique bounded linear operator $X$ on $L_{-}^2(S_{\mathfrak{R}_{\mathrm{II}}})$ such that for any $g,h\in L_{-}^2(S_{\mathfrak{R}_{\mathrm{II}}}),$
\beq \label{limAr} \lim_{r\to \infty}\langle A_{r}g,  h\rangle=B(g, h)= \langle X g,  h\rangle_{L_{-}^2(S_{\mathfrak{R}_{\mathrm{II}}})}.\eeq
Moreover, 
\begin{align*}
 \langle \mathcal{M}_{\phi_3}^* X \mathcal{M}_{\phi_3} g,  h\rangle_{L_{-}^2(S_{\mathfrak{R}_{\mathrm{II}}})}&=  \lim_{r\to \infty}\langle A_{r}\mathcal{M}_{\phi_3}g,  \mathcal{M}_{\phi_3}h\rangle_{L_{-}^2(S_{\mathfrak{R}_{\mathrm{II}}})}\\
 & =  \lim_{r\to \infty}\langle A_{r+1}g,  h\rangle_{L_{-}^2(S_{\mathfrak{R}_{\mathrm{II}}})}\\
 & \overset{\eqref{limAr}}{=}  \langle X g,  h\rangle_{L_{-}^2(S_{\mathfrak{R}_{\mathrm{II}}})} \end{align*}
 for all $g,h\in L_{-}^2(S_{\mathfrak{R}_{\mathrm{II}}}).$
 This  shows that $\mathcal{M}_{\phi_3}^* X \mathcal{M}_{\phi_3}=X.$

Now for any $g,h\in V,$ by Lemma \ref{lemma 2}, choose $r$ sufficiently large such that $\mathcal{M}_{\phi_3}^rg, \; \mathcal{M}_{\phi_3}^rh \in H_{-}^2(\mathfrak{R}_{\mathrm{II}}).$ Then

\begin{align*}
   &\Big\langle \mathcal{M}_{\phi_3}^{*r}A^*P_{H_{-}^2(\mathfrak{R}_{\mathrm{II}})}\mathcal{M}_{\phi_3}^{r}g, \mathcal{M}_{\phi_1}h\Big\rangle_{L_{-}^2(S_{\mathfrak{R}_{\mathrm{II}}})}\\
   &= \Big\langle A^* \mathcal{M}_{\phi_3}^{r}g, \mathcal{M}_{\phi_1}(\mathcal{M}_{\phi_3}^{r}h)\Big\rangle_{L_{-}^2(S_{\mathfrak{R}_{\mathrm{II}}})}\\
   &=\Big\langle \mathcal{T}_{\phi_1}^* A^* \mathcal{M}_{\phi_3}^{r}g, (\mathcal{M}_{\phi_3}^{r}h)\Big\rangle_{H_{-}^2(\mathfrak{R}_{\mathrm{II}})}\\
   &\overset{\eqref{brown-halmos10}}{=}\Big\langle \mathcal{T}_{\phi_3}^* A^* \mathcal{T}_{\phi_2} \mathcal{M}_{\phi_3}^{r}g, (\mathcal{M}_{\phi_3}^{r}h)\Big\rangle_{H_{-}^2(\mathfrak{R}_{\mathrm{II}})}\\
    &=\Big\langle  A^* P_{H_{-}^2(\mathfrak{R}_{\mathrm{II}})}\mathcal{M}_{\phi_2} \mathcal{M}_{\phi_3}^{r}g, \mathcal{M}_{\phi_3}^{r+1}h\Big\rangle_{H_{-}^2(\mathfrak{R}_{\mathrm{II}})}
  \end{align*}
\begin{align*}
   &\overset{\eqref{relation mathcal M}}{=}\langle A_{r+1}^*\mathcal{M}_{\phi_1}^*g, h\rangle_{L_{-}^2(S_{\mathfrak{R}_{\mathrm{II}}})}.
\end{align*}
Therefore, 
\begin{align*}
     \langle \mathcal{M}_{\phi_1}^* X^* g,  h\rangle_{L_{-}^2(S_{\mathfrak{R}_{\mathrm{II}}})} &= \lim_{r\to \infty}\langle A_{r+1}^*\mathcal{M}_{\phi_1}^*g, h\rangle_{L_{-}^2(S_{\mathfrak{R}_{\mathrm{II}}})}\\
     &\overset{\eqref{limAr}}{=} \langle X^* \mathcal{M}_{\phi_1}^* g,  h\rangle_{L_{-}^2(S_{\mathfrak{R}_{\mathrm{II}}})}.
\end{align*}
Since $V$ is dense in $L_{-}^2(S_{\mathfrak{R}_{\mathrm{II}}})$, 
it follows that $X \mathcal{M}_{\phi_1}=  \mathcal{M}_{\phi_1} X.$ A similar computation using the fact $A\mathcal{T}_{\phi_2}= \mathcal{T}_{\phi_1}^*A\mathcal{T}_{\phi_3}$ yields
$X \mathcal{M}_{\phi_2}=  \mathcal{M}_{\phi_2} X.$

We now prove $P_{H_{-}^2(\mathfrak{R}_{\mathrm{II}})} X|_{H_{-}^2(\mathfrak{R}_{\mathrm{II}})}=A.$ Let $g,h \in H_{-}^2(\mathfrak{R}_{\mathrm{II}}),$ then
\begin{align*}
 \langle  X g,  h\rangle_{L_{-}^2(S_{\mathfrak{R}_{\mathrm{II}}})}&\overset{\eqref{limAr}}{=} \lim_{r\to \infty}\Big\langle \mathcal{M}_{\phi_3}^{*r}AP_{H_{-}^2(\mathfrak{R}_{\mathrm{II}})}\mathcal{M}_{\phi_3}^{r}g,  h\Big\rangle_{L_{-}^2(S_{\mathfrak{R}_{\mathrm{II}}})}\\
 &=\lim_{r\to \infty}\Big\langle A \mathcal{M}_{\phi_3}^{r}g,   \mathcal{M}_{\phi_3}^{r}h\Big\rangle_{H_{-}^2(\mathfrak{R}_{\mathrm{II}})}\\
 &=\lim_{r\to \infty}\Big\langle A \mathcal{T}_{\phi_3}^{r}g,   \mathcal{T}_{\phi_3}^{r}h\Big\rangle_{H_{-}^2(\mathfrak{R}_{\mathrm{II}})}
 \overset{\eqref{brown-halmos10}}{=}\Big\langle  A g,   h\Big\rangle_{H_{-}^2(\mathfrak{R}_{\mathrm{II}})}.
\end{align*}
 Clearly, $\|A\| \leq \|X\|$ and from \eqref{limAr}, we get $\|X\| \leq \|A\|$. This completes the proof.
\end{proof}

The following result is a several variable analog of \cite[Corollary 12.7]{CJ2000}, we recall here for ease of reference (see also \cite[Lemma 3.4]{KMP2025}).

 \begin{lemma} \label{multiplication operator}
    Let $Y$ be a bounded linear operator on $L^2(S_{\mathbb E})$ such that $YM_{z_i}=M_{z_i}Y,\; i=1,2,3.$ Then there exists $\hat{\psi}\in L^\infty(S_{\mathbb E})$ such that $Y=M_{\hat{\psi}}$, multiplication operator on $L^2(S_\mathbb E)$ by $\hat{\psi}.$
\end{lemma}

Now, we are ready to prove the main result of this paper.

\begin{proof}[Proof of Theorem \ref{Brown-Halmos}:]
    Let $T_{\hat{\psi}}$ be a Toeplitz operator on $H^2(\mathbb E)$ associated with symbol $\hat{\psi} \in L^\infty(S_{\mathbb E}).$ For any $\hat{f},\hat{g}\in H^2(\mathbb E),$
    \beqn \inp{T_{\hat{\psi}}T_{z_1}\hat{f}}{\hat{g}}_{H^2(S_{\mathbb E})} &=& \inp{M_{\hat{\psi}}z_1\hat{f}}{\hat{g}}_{L^2(S_\mathbb E)}\\
    &=& \int_{S_\mathbb E} \hat{\psi}(z) \bar{z}_2 z_3 \hat{f}(z) \overline{\hat{g}(z)} d\Theta_{\mathbb E}(z)\\
    &=& \inp{M_{\hat{\psi}} z_3 \hat{f}}{z_2 \hat{g}}_{L^2(S_\mathbb E)}\\
    &=& \inp{T_{\hat{\psi}} T_{z_3}\hat{f}}{T_{z_2}\hat{g}}_{H^2(\mathbb E)}= \inp{T_{z_2}^*T_{\hat{\psi}} T_{z_3}\hat{f}}{\hat{g}}_{H^2(\mathbb E)}. \eeqn
   Similarly, one can verify that $T_{\hat{\psi}} T_{z_2}=T_{z_1}^* T_{\hat{\psi}} T_{z_3}$, and $T_{z_3}^* T_{\hat{\psi}} T_{z_3}=T_{\hat{\psi}}.$
    
    Conversely, let $T$ be a bounded linear operator on $H^2(\mathbb E)$ satisfying the conditions $T T_{z_1}=T_{z_2}^* T T_{z_3}$, $T T_{z_2}=T_{z_1}^* T T_{z_3}$, and $T_{z_3}^* T T_{z_3}=T.$ Note that the operators $\Psi$ and $\tilde{\Psi}$ given in \eqref{psihardyhardy} and \eqref{eq 12} are unitaries and $\Psi=\tilde{\Psi}|_{H^2(\mathbb E)}.$
     We claim that the bounded linear operator $A=\Psi\; T\; \Psi^*$, defined on $H^2_{-}(\mathfrak{R}_{\mathrm{II}})$ satisfies the following conditions: $$A\mathcal{T}_{\phi_1}=\mathcal{T}_{\phi_2}^* A \mathcal{T}_{\phi_3},\; A\mathcal{T}_{\phi_2}=\mathcal{T}_{\phi_1}^* A \mathcal{T}_{\phi_3}, \text{and}\,\, \mathcal{T}_{\phi_3}^*A\mathcal{T}_{\phi_3}= A.$$ Indeed, 
    \beqn A \mathcal{T}_{\phi_1}&=& \Psi\; T\;  \Psi^* \; \Psi \; T_{z_1} \Psi^*= \Psi\; T T_{z_1} \Psi^*\\
    &=&\Psi\; T_{z_2}^* T T_{z_3} \Psi^*\\
    &=& \Psi\; T_{z_2}^* \Psi^*\; \Psi T \Psi^*\; \Psi T_{z_3} \Psi^*\\
    &=& \mathcal{T}_{\phi_2}^* A \mathcal{T}_{\phi_3}.
    \eeqn
    Similarly, the other equalities can be derived. By Lemma \ref{key lemma}, there exists a bounded linear operator $X$ on $L^2_{-}(S_{\mathfrak{R}_{\mathrm{II}}})$ such that $P_{H_{-}^2(\mathfrak{R}_{\mathrm{II}})}X|_{H_{-}^2(\mathfrak{R}_{\mathrm{II}})}=A,$ $ X\mathcal{M}_{\phi_i}=\mathcal{M}_{\phi_i}X$ for $i=1,2,3$, and $\|X\|=\|A\|.$

If $X'=\tilde{\Psi}^* X \tilde{\Psi}$ is a bounded linear operator on $L^2(S_\mathbb E),$ then for all $i=1,2,3,$ we get \beqn X'M_{z_i}&=& \tilde{\Psi}^* X \;\tilde{\Psi} \;\tilde{\Psi}^* \mathcal{M}_{\phi_i} \tilde{\Psi}=\tilde{\Psi}^* X \mathcal{M}_{\phi_i} \tilde{\Psi}\\
    &=& \tilde{\Psi}^* \mathcal{M}_{\phi_i} X \tilde{\Psi}= \tilde{\Psi}^* \mathcal{M}_{\phi_i} \;\tilde{\Psi} \;\tilde{\Psi}^* X \tilde{\Psi}\\
    &=& M_{z_i} X'.
    \eeqn
For any $\hat{f},\; \hat{g}\in H^2(S_\mathbb E)$, 
    \beqn \inp{P_{H^2(\mathbb E)} X' |_{H^2(\mathbb E)} \hat{f}}{\hat{g}}_{H^2(\mathbb E)}&=& \inp{X' \hat{f}}{\hat{g}}_{L^2(S_\mathbb E)}\\
    &=& \inp{X \tilde{\Psi} \hat{f}}{ \tilde{\Psi} \hat{g}}_{L^2_{-}(S_{\mathfrak{R}_{\mathrm{II}}})}\\
    &=& \inp{X \Psi \hat{f}}{ \Psi \hat{g}}_{L^2_{-}(S_{\mathfrak{R}_{\mathrm{II}}})}\\
    &=& \inp{P_{H_{-}^2(\mathfrak{R}_{\mathrm{II}})}X|_{H_{-}^2(\mathfrak{R}_{\mathrm{II}})}\Psi \hat{f}}{ \Psi \hat{g}}_{H^2_{-}(\mathfrak{R}_{\mathrm{II}})}\\
    &=& \inp{\Psi^* A \Psi \hat{f}}{\hat{g}}_{H^2(\mathbb E)}= \inp{T \hat{f}}{\hat{g}}_{H^2(\mathbb E)}.
    \eeqn
This shows that $P_{H^2(\mathbb E)} X' |_{H^2(\mathbb E)}=T$ and $X' M_{z_i}=M_{z_i} X'$ for all $i=1,2,3.$ The theorem is concluded by applying Lemma \ref{multiplication operator}.
\end{proof}

\section{Compact Toeplitz operator on the tetrablock}

Brown and Halmos observed that every compact Toeplitz operator on the Hardy space of the unit disc is the zero operator (see the Corollary on page 94 of \cite{BH1964}). 
This conclusion for the polydisc, the unit ball, and the classical Cartan domains follows from \cite[Theorem 3.2]{MSS2018}, \cite[Theorem 3.3]{DE2011}, and \cite{KMP2025}, respectively. In this section, we establish that the same phenomenon occurs for Toeplitz operators on the Hardy space of the tetrablock.
To prove this, we adopt the approach of Brown and Halmos from \cite{BH1964} (see also \cite[Theorem 4.20]{GR2024}).

Recall that $H^2_{-}(\mathfrak{R}_{\mathrm{II}})=\{f\in H^2(\mathfrak{R}_{\mathrm{II}}): f\circ\sigma=-f\}.$ Clearly, $1 \notin H^2_{-}(\mathfrak{R}_{\mathrm{II}}).$ 
Let $\mathrm{Hom}^-(n)$ denote the space of homogeneous polynomials of degree $n$ in  $H^2_{-}(\mathfrak{R}_{\mathrm{II}}).$ Then $\mathrm{Hom}^-(n)$ is orthogonal to $\mathrm{Hom}^-(\ell)$ for $n \neq \ell.$
Also, note that $\mathrm{Hom}^-(n)$ is spanned by $$\left \{z_1^{\alpha_1} z_2^{\alpha_2} z_3^{2m-1}: \alpha_1,\alpha_2\in \mathbb Z_+, \alpha_1+\alpha_2=n-(2m-1), \; m=1,\ldots,k \right\},$$ where $k=\frac{n}{2}$ if $n$ is even and $k=\frac{n+1}{2}$ if $n$ is odd. If $d_n^-$ denotes the dimension of $\mathrm{Hom}^-(n)$ then $d_n^-=\frac{(n+1)^2}{4}$ whenever $n$ is odd, and $d_n^-=\frac{n(n+2)}{4}$ whenever $n$ is even.
\begin{lemma}\label{orthonormalbasis}
    For $n\in \mathbb N,$ there is an orthonormal basis $\mathcal{E}_n$ of $\mathrm{Hom}^-(n)$ in $H^2_{-}(\mathfrak{R}_{\mathrm{II}})$ such that $\phi_3 \mathcal{E}_n =\{\phi_3  f: f\in \mathcal{E}_n\}\subseteq\mathcal{E}_{n+2}$, where $\phi_3(z_1,z_2,z_3)=z_1z_2-z_3^2.$ 
\end{lemma}
\begin{proof}
    
Let $\mathcal{E}_1=\{e_1^1\}$ and $\mathcal{E}_2=\{e^2_1,e^2_2\}$ be orthonormal bases of $\mathrm{Hom}^-(1)$ and $\mathrm{Hom}^-(2),$ respectively.  
We begin by considering a spanning set of $\mathrm{Hom}^-(3)$ that contains $\phi_3 \mathcal{E}_1,$ given by:
$$\{(z_1z_2-z_3^2)e_1^1,\; (z_1z_2+z_3^2)e_1^1,\; z_1^2 e_1^1,\; z_2^2e_1^1\}.$$ Clearly, $\{\phi_3 e_1^1\}$ is an orthonormal set. 
By applying the Gram-Schmidt orthogonalization process to the above spanning set, followed by normalization, yields an orthonormal basis $\mathcal{E}_3=\{e_i^3\}_{i=1}^{d_3^-}$  of $\mathrm{Hom}^-(3)$, where $e^3_1=\phi_3 e_1^1.$ 

Similarly, consider the following spanning set of $\mathrm{Hom}^-(4)$ containing $\phi_3 \mathcal{E}_2:$ 
\beq \label{spanningset4}  \{(z_1z_2-z_3^2)e^2_i,\; (z_1z_2+z_3^2)e^2_i,\; z_1^2 e^2_i,\; z_2^2e^2_i: i=1,2\}.\eeq 
Denote $e^4_1:=\phi_3 e_1^2$, and $e^4_2:= \phi_3 e_2^2$. The set $\{e^4_1, e^4_2\}$ forms an orthonormal subset of $\mathrm{Hom}^-(4).$ We apply the Gram-Schmidt orthogonalization process to the spanning set in \eqref{spanningset4}, taking $e_1^4$ and $e_2^4$ as the initial orthonormal vectors, and subsequently normalizing, we get an orthonormal basis $\mathcal{E}_4$ of $\mathrm{Hom}^-(4)$ such that $\phi_3 \mathcal{E}_2 \subseteq \mathcal{E}_4.$

Proceeding inductively, we obtain an orthonormal basis $\mathcal{E}_n=\{e^n_i\}_{i=1}^{d_n^-}$ of $\mathrm{Hom}^-(n),$ and an orthonormal basis $\mathcal{E}_{n+2}=\{e^{n+2}_i\}_{i=1}^{d_{n+2}^-}$ of $\mathrm{Hom}^-(n+2),$ such that $e_i^{n+2}=\phi_3 e_i^n$ for $i=1,\ldots, d_n^-.$ It follows that $\phi_3\mathcal{E}_n \subseteq \mathcal{E}_{n+2}.$ This completes the proof of the lemma.

\end{proof}

Now, as an application of Lemma \ref{key lemma}, we prove the following result.
\begin{theorem}
    The only compact Toeplitz operator on $H^2(\mathbb E)$ is the zero operator.
\end{theorem}
\begin{proof}
  Let $\hat{u} \in L^\infty(S_\mathbb E).$ Assume that the Toeplitz operator $T_{\hat{u}}$ on $H^2(\mathbb E)$ is compact. If $u=\hat{u}\circ \phi$ 
then it follows from \eqref{toeplitzrelation} that the Toeplitz operator $\mathcal{T}_{u}$ defined on  $H^2_{-}(\mathfrak{R}_{\mathrm{II}})$ is unitarily equivalent to $T_{\hat{u}}.$  Therefore,  $\mathcal{T}_{u}$ is compact.
 
   Let $\mathcal{E}=\bigcup_{n=1}^\infty \mathcal{E}_n$ be the orthonormal basis of $H^2_{-}(\mathfrak{R}_{\mathrm{II}}),$ where $\mathcal{E}_n= \{e^n_i: 1\leq i \leq d_n^-\}$ is the orthonormal basis of $\mathrm{Hom}^-(n)$ obtained in Lemma \ref{orthonormalbasis}. 
   Note that for any $e^n_i\in \mathcal{E}$, $\mathcal{T}_{\phi_3} e^n_i= e^{n+2}_i.$
   If $e^n_i,\; e^m_j\in \mathcal{E}$ and $r\in \mathbb N$ then by Lemma \ref{key lemma}, we have $$\inp{\mathcal{T}_u e^n_i}{e^m_j}= \inp{\mathcal{T}_{\phi_3}^{r*} \mathcal{T}_u \mathcal{T}_{\phi_3}^r e^n_i}{e^m_j}= \inp{\mathcal{T}_u e^{n+2r}_i}{e^{m+2r}_j}.$$
  Since $e_i^\ell$ converges weakly to $0$ as $\ell \to \infty$ and $\mathcal{T}_u$ is compact, we get  $$|\inp{\mathcal{T}_u e^n_i}{e^m_j}|  \leq \|\mathcal{T}_u e^{n+2r}_i\|  \to 0 \; \mbox{as } r\to \infty.$$ Therefore, $\inp{\mathcal{T}_u e^n_i}{e^m_j}=0$. Since $e^n_i,\; e^m_j$ are chosen arbitrarily from $\mathcal{E}$, $\mathcal{T}_u=0$. This shows that  $u=\hat{u}\circ \phi=0$ a.e on $S_{\mathfrak{R}_{\mathrm{II}}}.$ Consequently $\hat{u}=0$ a.e on $S_{\mathbb E}.$ This completes the proof.
\end{proof}
   
\subsection*{Acknowledgments}
The author, S. Jain, gratefully acknowledges Prof. Arup Chattopadhyay for the financial support provided through his project (Ref No. MATHSPNSERB01119xARC002).
The work of S. Kumar was supported in the form of MATRICS grant (Ref No. MTR/2022/000457) of Science and Engineering Research Board (SERB). 
Support for the work of M. K. Mal was provided in the form of a Prime Minister's Research Fellowship (PMRF / 2502827). 
Support for the work of P. Pramanick was provided by the Department of Science and Technology (DST) in the form of the Inspire Faculty Fellowship (Ref No. DST/INSPIRE/04/2023/001530).

\end{document}